\documentclass[a4paper,12pt]{article}
\usepackage{amsrefs,amstext,amsmath,amsthm,amsfonts,amssymb,enumerate,amscd,makeidx,mathrsfs}
\usepackage[refpage,intoc]{nomencl}
\usepackage{hyperref}
\usepackage[dvipdf]{graphicx}
\usepackage[margin=3 true cm]{geometry}
\usepackage{todonotes}
\usepackage{xspace}

\newtheorem{theorem}{Theorem}[section]
\newtheorem{proposition}[theorem]{Proposition}
\newtheorem{notation}[theorem]{Notation}
\newtheorem{construction}[theorem]{Construction}
\newtheorem{lemma}[theorem]{Lemma}

\newtheorem{example}[theorem]{Example}
\newtheorem{definition}[theorem]{Definition}
\newtheorem{remark}[theorem]{Remark}

\newtheorem{convention}[theorem]{Convention}

\newtheorem*{lma}{Lemma}

\newtheorem*{prop}{Proposition}
\newtheorem*{nota}{Notation}

\def\mathbi#1{\textbf{\em #1}}

\makenomenclature

\title{The Bridge Lemmas between Equivalent Fell Bundles and its Applications}
\author{Weijiao He}
\date{}

\begin{document}

\maketitle

\begin{abstract}
In this paper, we prove that the induced representation theories of two equivalent Fell bundles are essentially identical; and we apply our results to carry the induced representation theory and imprimitivity theorems of saturated Fell bundles to arbitrary Fell bundles.   
\end{abstract}

\section*{Introduction}
A Fell bundle over a locally compact group $G$ is a continuous bundle $\mathscr{B} \to G$ of
Banach spaces $(B_t)_{t \in G}$ together with continuous multiplications $B_s \times B_t \to B_{st}, (a,b) \mapsto ab$ 
, and involutions $B_s \to B_{s^{-1}}$, satisfying properties similar
to those valid for a $C^{\ast}$-algebra, like the positivity $a^{\ast}a \geq 0$ and the $C^{\ast}$-axiom
$\|a^{\ast}a\|=\|a\|^2$. Fell bundles generalize partial actions of groups. We refer the reader to \cite[\S VIII]{MR936629} for details.

The notion of equivalent Fell bundles, which is a generalization of Morita equivalence of $C^{\ast}$-algebras, is originally presented by Muhly and Williams in \cite{MR2446021}, Abadie in \cite{MR1957674} and Abadie and P\'{e}rez in \cite{MR2529875} with varied forms, and is systematically studied by Abadie and Ferraro in \cite{MR3959060}, and Abadie, Buss and Ferraro in \cite{vjnj}. 

In \cite{MR3959060} and \cite{vjnj}, one of the important results is that an arbitrary Fell bundle is equivalent to a saturated Fell bundle. Based on this result, as in the context of Morita-equivalent $C^{\ast}$-algebras, we shall be very interested in to investigate the representation-theoretic equivalence between equivalent Fell bundles; and by the aid of this representation-theoretic equivalence we may try to generalize as many results on the representation theory of saturated Fell bundles as possible to the context of arbitrary Fell bundles.

In the present paper we deal with the induced representation theories of equivalent Fell bundles. 

We  roughly sketch the problem we deal with as following. Before preceding, we refer the reader to \cite{vjnj} or \ref{zdjfdcdedcz} to see the definition of equivalence between Fell bundles. Given two equivalent Fell bundles $\mathscr{B}^{(1)}$ and $\mathscr{B}^{(2)}$ over a locally compact group $G$ implemented by a $\mathscr{B}^{(1)}$-$\mathscr{B}^{(2)}$ equivalence bundle $\mathcal{L}$, by \cite{MR3959060} and \cite{vjnj}, $\mathscr{L}(\mathcal{L})$, which is consisting of the continuous cross-sections of $\mathcal{L}$ vanishing outside compact subsets, is a $\mathscr{L}(\mathscr{B}^{(1)})$-$\mathscr{L}(\mathscr{B}^{(2)})$ imprimitivity bimodule (see Definition \ref{sdkjlzdklczoe}); and the full $C^{\ast}$-algebras $C^{\ast}(\mathscr{B}^{(1)})$ and $C^{\ast}(\mathscr{B}^{(2)})$ are pre-Morita equivalent implemented by $\mathscr{L}(\mathcal{L})$. Now let $H$ be a closed subgroup of $G$, and $\mathscr{B}_H^{(1)}$, $\mathscr{B}_H^{(2)}$ and $\mathcal{L}_H$  the restricted bundles; with the  bimodule structure inherited from $\mathscr{L}(\mathcal{L})$, $\mathscr{L}(\mathcal{L}_H)$ is a $\mathscr{L}(\mathscr{B}^{(1)})$-$\mathscr{L}(\mathscr{B}^{(2)})$ imprimitivity bimodule. A natural question is the
following: for a given $\ast$-representation $S$ of $\mathscr{B}_H^{(1)}$, (a) is it possible to induce $S$ to a $\ast$-representation of $\mathscr{B}_H^{(2)}$ (which we denote by $T$) via $\mathscr{L}(\mathcal{L}_H)$? (b) Is ${\rm{Ind}}_{\mathscr{B}^{(2)}_{H} \uparrow \mathscr{B}^{(2)}}(T)$ corresponding to ${\rm{Ind}}_{\mathscr{B}^{(1)}_{H} \uparrow \mathscr{B}^{(1)}}(S)$ via the pre-Morita equivalence implemented by $\mathscr{L}(\mathcal{L})$? If we have affirmative answers, the induced representation theory of $\mathscr{B}^{(1)}$ can be ``carried'' into $\mathscr{B}^{(2)}$ by the equivalence between $\mathscr{B}^{(1)}$ and $\mathscr{B}^{(2)}$; so we will know the induced representation theory of one of them if we know it of the other one.  




The structure of this paper is as follows.

In Section 1, we review the basic knowledges used in this paper.

In Section 2, we prove the three Bridge Lemmas which give affirmative answers to the questions proposed in the preceding paragraph from different perspectives.   

In Section 3, we give the applications of the Bridge Lemmas.  Let $\mathscr{B}$ be a Fell bundle over a locally compact group $G$, and let $H$ a closed subgroup of $G$; our first application is to prove that any $\ast$-representation of a restricted bundle is $\mathscr{B}$-positive, which is proved by Ferraro in \cite{ccc} by a different and more self-contained way. The second application is to investigate the relationship between the two important imprimitivity theorems on Fell bundles; and Theorem \ref{ewridskjodsklcxwqweqwe} gives a complete answer.
\section{Preliminaries}
 \begin{nota}
\rm Throughout this paper, $A$ and $B$ will denote $\ast$-algebras; $G$ is a  locally compact group; and $H$ is a closed subgroup $G$. $\Delta$ and $\delta$ are the Haar measures of $G$ and $H$. We choose once a fixed rho-function $\rho$. We use the symbol $\mathscr{B}$ to denote a Fell bundle over $G$; and $\mathscr{B}_H$ is the restricted bundle. We use the symbol $\mathcal{O}(X)$ to denote the space of the bounded operators on the Hilbert space $X$.
 \end{nota}
 The objective of this section is to give a brief review of the basic knowledge about induced representations.

\subsection{Induced Representation and Morita Equivalence: Abstract Version}

The materials of this part are based on \cite{MR936629} and \cite{MR1634408}.










\begin{definition}
\rm A $right$ $B$-$rigged$ $left$ $A$-$module$ (resp. $left$ $A$-$rigged$ $right$ $B$-$module$) is a linear space $L$ which is both a
left $A$-module and a right $B$-module, together with a map $[\ , \ ]$ from $L \times L$ into $B$ (resp. $A$) such that

$(i)$ $[s, t ]$ is linear in $t$ (resp. $s$) and conjugate-linear in $s$ (resp. $t$) $(s, t \in L)$;

$(ii)$ $[ t , s]= [s, t]^{\ast}$ $(s, t \in L)$;

$(iii)$ $[s, tb ]= [s, r ]b$ (resp. $[bs, t]= b[s, r ]$) for all $s, t$ in $L$ and $b$ in $B$ (resp. $A$);

$(iv)$ $[as,t]=[s,a^{\ast}t]$ (resp. $[sa,t]=[s, ta^{\ast}]$) for all $a$ in $A$ (resp. $B$) and $s, t \in L$.

\end{definition}

\begin{construction}\label{eriodoadoiascerer}
\rm  ($\mathbf{Operator \ inner \ product }$) Let $L$ be a linear space. A map $P: L \times L \to \mathcal{O}(X)$ is called $an \ operator \ inner \ product$ on $L$ if: (i) $V_{s,t}$ is linear in $s$ and conjugate linear in $t$; (ii) $P$ is $positive$ in the sense that
\begin{equation*}
\sum_{i,j=1}^n (V_{t_i,t_j}\xi_i,\xi_j) \geq 0
\end{equation*}
for any positive integer $n$, any $t_1,...t_n$ in $L$, and any $\xi_1,...\xi_n$ in $X$. In this case, we define an inner product $(\ , \ ): (L \otimes X) \times (L \otimes X) \to \mathbb{C}$ by
\begin{equation*}
(r  {\otimes} \xi, t  {\otimes} \eta)=(S_{[t,r]}\xi,\eta) \ \ \ \ (r, s \in L; \ \xi, \ \eta \in X). 
\end{equation*}
\begin{nota}
\rm We denote the Hilbert space resulting from the quotient and completing of $L \otimes X$ with respect to $(\ , \ )$ by $\mathscr{X}(P)$; and we denote the image of $r \otimes \xi$ in $\mathscr{X}(P)$ by $r \widetilde{\otimes} \xi$  under the quotient map. 
\end{nota}
\end{construction}
\begin{construction}\label{serijcdkjdiucjifdipier}
\rm ($\mathbf{Abstract \ Induced \ Representations }$) Let $\langle L, [\ , \ ]\rangle$ be a right $B$-rigged left $A$-module. Let $S$ be a non-degenerate $\ast$-representation of $B$ acting on a Hilbert space $X$. We define $P_S: L \times L \to \mathcal{O}(X)$ by 
\begin{equation*}
P_S(r,s)=S_{[s,r]} \ \ \ \ (r, s \in L).
\end{equation*}
We say that $S$ or $P_S$ is $\langle L, [\ , \ ]\rangle$-$\mathbi{positive}$ if $P_S$ is an operator inner product on $L$; and $S$ or $P_S$ is $A$-$\mathbi{bounded}$ if the sense that the unique linear map $T_a$ defined on $L \otimes X(S)$ satisfying
\begin{equation*}
r \otimes \xi \mapsto (a \cdot r) \otimes \xi \ \ \ \ (r \in L,  \xi \in X(S))
\end{equation*}
can be extended to a bounded operator  on $\mathscr{X}(P_S)$ which we still denote by $T_a$ for each $a \in A$. If $S$ or $P_S$ is $\langle L, [\ , \ ]\rangle$-positive and bounded, then $a \mapsto T_a$ is a $\ast$-representation of $A$. We say that $S$ is $\mathbi{induceble}$ $\mathbi{via}$ $\langle L, [\ , \ ]\rangle$ and $T$ is $\mathbi{induced  from}$ $S$ $\mathbi{via}$ $\langle L, [\ , \ ]\rangle$; and we denote $T={\rm{Ind}}(S; \langle L, [\ , \ ]\rangle)$.
\end{construction}
\begin{notation}
\rm Sometimes we denote $L \otimes_S X(S)$ instead of $\mathfrak{X}(P_S)$.
\end{notation}

\begin{example}\label{eijaioadocareadas}
\rm If $p: B \to A$ is a conditional expectation, then we can define a map $[ \ , \ ]^p: A \times A \to B$ by 
\begin{equation*}
[a,b]^p=p(a^{\ast}b) \ \ \ \ (a, b \in B);
\end{equation*}
and it is readily seen that $\langle A, [\ , \ ]^p \rangle$ is a right $B$-rigged left $A$-module. If a $\ast$-representation $S$ of $B$ is inducible to $A$ via $\langle A, [\ , \ ]^p \rangle$, we also say that $S$ $\mathbi{is  inducible}$ $\mathbi{via}$ $p$. 
\end{example}

\begin{definition}\label{sdkjlzdklczoe}
\rm An $A$, $B$ $imprimitivity$
$bimodule$ is a system $\langle L,\  _A[ \ , \ ], [ \ , \ ]_B \rangle$, where

(i) $L$ is both a left $A$-module and a right $B$-module,

(ii) $\langle L, [\ , \ ]_B \rangle$ is a right $B$-rigged left $A$-module,

(iii) $\langle L, \ _A[\ ,\ ] \rangle$ is a left $A$-rigged right $B$-module,

(vi) the associative relation
\begin{equation*}
_A[r,s]t=r[s,t]_B  \ \ \ \ (r, s \in L)
\end{equation*}
holds. 

If $A$ and $B$ are $C^{\ast}$-algebras, we say that $A$ and $B$ are $Morita$ $equivalent$ if the following two additional conditions hold:

(vii) $\{_A[r,s]: r, s \in L\}$ is dense in $A$ and $\{[r,s]_B: r, s \in L\}$ is dense in $B$,

(viii) $L$ is complete with respect to the norm $\| \ \cdot \ \|_A$ defined by $\|r\|=\|_A[r,r]\|^{1/2}$, and is complete with respect to the norm $\| \ \cdot \ \|_B$ defined by $\|r\|=\|[r,r]_B\|^{1/2}$.

\end{definition}

\begin{definition}
\rm Let $A$ and $B$ be two $C^{\ast}$-algebras; and let $A_0 \subset A$ and $B_0 \subset B$ be dense $\ast$-subalgebras;  $\langle L,\  _{A_0}[ \ , \ ], [ \ , \ ]_{B_0} \rangle$ is an $A_0$, $B_0$ imprimitivity bimodule. We say that $A$ and $B$ are $pre$-$Morita$ $equivalent$ if the following two conditions hold:

(vii') The linear spans of $\{_A[r,s]: r, s \in L\}$ and $\{[r,s]_B: r, s \in L\}$ are dense in $A$ and $B$ respectively;

(viii') $[ar, ar]_{B_0} \leq \|a\|^2 [r,r]_{B_0}$;   $_{A_0}[rb, rb] \leq \|b\|^2 \ _{A_0}[r,r]$ ($r \in L; \ a \in A_0; \ b \in B_0$).
\end{definition}
 
 \begin{convention}\label{fdsickjldfsofxc}
\rm Thanks to \cite[Proposition 3.12]{MR1634408}, if $A$ and $B$ are pre-Morita equivalent implemented by an $A_0$, $B_0$ imprimitivity bimodule $\langle L,\  _{A_0}[ \ , \ ], [ \ , \ ]_{B_0} \rangle$, then we can complete $\langle L,\  _{A_0}[ \ , \ ], [ \ , \ ]_{B_0} \rangle$ such that $A$ and $B$ are Morita equivalent implemented by this completion. $\mathbf{For \ the \ sake \ of \ convenience, \ we \ will \ say \ that}$ $\mathbi{A}$ $\mathbf{and}$ $\mathbi{B}$ $\mathbf{are \ Morita \ equivalent \ implemented \ by}$ $\langle L,\  _{A_0}[ \ , \ ], [ \ , \ ]_{B_0} \rangle$.
\end{convention}

The following lemma is easy to be verified:
\begin{lemma}\label{esrijodzjodfiojcde}
Let $A$ and $B$ be two $\ast$-algebras;  $\langle L,\  _A[ \ , \ ], [ \ , \ ]_B \rangle$ is an $A$, $B$ imprimitivity
bimodule. Then the $C^{\ast}$-completion of $A$ and $B$ are Morita equivalent implemented by $\langle L,\  _{A }[ \ , \ ], [ \ , \ ]_{B } \rangle$ if either of the following conditions holds:

$(i)$ Any non-degenerate $\ast$-representation of $A$ is inducible to a non-zero $\ast$-$represe$-$ntation$ of $B$ via $L$, and any non-degenerate $\ast$-representation is induced from a $\ast$-representation of $A$;

$(ii)$ Any non-degenerate $\ast$-representation of $A$ is inducible to a non-zero $\ast$-$representation$ of $B$ via $L$, and any non-degenerate $\ast$-representation of $B$ is inducible to a non-zero $\ast$-representation of $A$ via $L$.
\end{lemma}

\subsection{Induced Representations and Morita Equivalence: Examples of Fell Bundles}

\begin{definition}
\rm (\cite[XI.8.6]{MR936629}) Let $S$ be a $\ast$-representation of $\mathscr{B}_H$. If $S_{p(f^{\ast} \ast f)} \geq 0$ for all $f \in \mathscr{L}(\mathscr{B})$, then we say that $S$ is $\mathscr{B}$-$positive$.
\end{definition}

\begin{construction}\label{seriodiofdsiofdsoierw}
\rm ($\mathbf{The \ concrete \ construction \ of}$ ${\rm{Ind}}_{\mathscr{B}_H \uparrow \mathscr{B}}(S)$) (\cite[\S XI.9]{MR936629}) Let $S$ be a non-degenerate $\ast$-representa-tion of $\mathscr{B}_H$. For each $\alpha \in G/H$ we denote the algebraic direct sum $\sum_{x \in \alpha}^{\oplus}(B_x \otimes X(S))$ of the algebraic tensor products $B_x \otimes X(S)$ by $Z_{\alpha}$. We introduce into $Z_{\alpha}$ the conjugate-bilinear form $( \ ,\ )_{\alpha}$ by
\begin{equation*}
(b \otimes \xi, c \otimes \eta)_{\alpha}=(\rho(x)\rho(y))^{-1/2}(S_{c^{\ast}b}\xi, \eta)
\end{equation*}
($x, \ y \in \alpha; \ b \in B_x; \ c \in B_y; \ \xi, \ \eta \in X(S)$). $\mathbf{One \ can \ prove \ that}$ $( \ , \ )_{\alpha}$ $\mathbf{is \ positive}$ $\mathbf{for \ all}$ $\alpha$ $\mathbf{if \ and \ only \ if}$ $\mathbi{S}$ $\mathbf{is}$ $\mathscr{B}$-$\mathbf{positive}$. Thus if $S$ is $\mathscr{B}$-positive we can form a Hilbert space $Y_{\alpha}$ by factoring out from $Z_{\alpha}$ the null space of $( \ , \ )_{\alpha}$ and completing, and a Hilbert bundle $\mathscr{Y}$ over $G/H$ with the fibers $Y_{\alpha}$; we denote the bundle space of $\mathscr{Y}$ by $Y$. Let $\kappa_{\alpha}: Z_{\alpha} \to Y_{\alpha}$ be the quotient map for each $\alpha \in G/H$. For each $c \in B_x$ there is a continuous map $\tau_{c}: Y \to Y$ defined by
\begin{equation*}
\tau_c(\kappa_{\alpha}(b \otimes \xi))=\kappa_{x \alpha}(cb \otimes \xi) \ \ \ \ (b \in B_{y}, y \in \alpha; \xi \in X(S));
\end{equation*}
and the following map
\begin{equation*}
(\alpha \mapsto f(\alpha)) \mapsto (f: \alpha \mapsto \tau_bf(x^{-1}\alpha)) \ \ \ \ (f \in \mathscr{L}_2(G/H;\rho^{\sharp}; \mathscr{Y}))
\end{equation*}
is a bounded operator on the Hilbert space $\mathscr{L}_2(G/H; \rho^{\sharp}; \mathscr{Y})$, which we denote by ${\rm{Ind}}_{\mathscr{B}_H \uparrow \mathscr{B}}(S)(b)$. One can prove that $b \mapsto {\rm{Ind}}_{\mathscr{B}_H \uparrow \mathscr{B}}(S)(b)$ is a non-degenerate $\ast$-representation of $\mathscr{B}$.

\end{construction}

\begin{construction}\label{rseidierdcdewer}
\rm (\cite[\S XI.9]{MR936629}) In order to construct ${\rm{Ind}}_{\mathscr{B}_H \uparrow \mathscr{B}}(S)$ in the abstract approach as defined in Construction \ref{serijcdkjdiucjifdipier},  we need to make $\mathscr{L}(\mathscr{B})$ a right $\mathscr{L}(\mathscr{B}_H)$-rigged left $\mathscr{L}(\mathscr{B})$-module. Let $p: \mathscr{L}(\mathscr{B}) \to \mathscr{L}(\mathscr{B}_H)$ be the conditional expectation given by
\begin{equation}\label{esiodkdoieodaxs}
p(f)(h)=f(h)(\Delta(h))^{1/2}(\delta(h))^{-1/2} \ \ \ \ (f \in \mathscr{L}(\mathscr{B}); h \in H).
\end{equation}
We define $[ \ , \ ]_{\mathscr{L}(\mathscr{B}_H)}=[ \ , \ ]^p$ (see Example \ref{eijaioadocareadas}); and we define the left action of  $\mathscr{L}(\mathscr{B})$ on  $\mathscr{L}(\mathscr{B})$ as the convolution of cross-sections; and the right action of  $\mathscr{L}(\mathscr{B}_H)$ on  $\mathscr{L}(\mathscr{B})$ by 
\begin{equation*}
(f\phi)(x)=\int_Hf(xh)\phi(h^{-1})\Delta(h)^{1/2}\delta(h)^{1/2}dvh 
\end{equation*}
$ (f \in \mathscr{L}(\mathscr{B}); \phi \in \mathscr{L}(\mathscr{B}_H); h \in H)$. By \cite[\S XI.8]{MR936629} $\langle \mathscr{L}(\mathscr{B}), [ \ , \ ]_{\mathscr{L}(\mathscr{B}_H)} \rangle$ is a right $\mathscr{L}(\mathscr{B}_H)$-rigged left $\mathscr{L}(\mathscr{B})$-module.

\begin{prop}
$($\cite[XI.9.26]{MR936629}$)$ If $S$ is a non-degenerate $\ast$-representation of $\mathscr{B}_H$, then $S$ is inducible to $\mathscr{B}$ via $\langle \mathscr{L}(\mathscr{B}), [ \ , \ ]_{\mathscr{L}(\mathscr{B}_H)} \rangle$; and we have
\begin{equation*}
{\rm{Ind}}_{\mathscr{B}_H \uparrow \mathscr{B}}(S) \simeq {\rm{Ind}}(S; \langle \mathscr{L}(\mathscr{B}), [ \ , \ ]_{\mathscr{L}(\mathscr{B}_H)} \rangle).
\end{equation*} 
\end{prop}
\end{construction}

\subsection{Equivalence of Fell Bundles}\label{zdjfdcdedcz}
This section is a short summary of the results of \cite{vjnj}. Let $\mathscr{B}^{(1)}$ and $\mathscr{B}^{(2)}$ be two Fell bundles over $G$. 
\begin{definition}
\rm Let $\mathcal{L}$ be Banach bundle $\mathcal{L}=\{L_x\}_{x \in G}$ over $G$. We say that $\mathcal{L}$ is a $\mathscr{B}^{(1)}$-$\mathscr{B}^{(2)}$-$equivalence$ $bundle$ if there are continuous maps
\begin{equation}\begin{split}
& [ \ , \ ]_{\mathscr{B}^{(1)}}: \mathcal{L} \times \mathcal{L} \to \mathscr{B}^{(1)}, \ \mathcal{L}  \times \mathscr{B}^{(1)} \to \mathscr{B}^{(1)};
\\& _{\mathscr{B}^{(2)}}[ \ , \ ]: \mathcal{L} \times \mathcal{L} \to \mathscr{B}^{(2)}, \ \mathscr{B}^{(1)} \times \mathcal{L}  \to  \mathscr{B}^{(1)}, 
\end{split}\end{equation}
such that:

$(i)$ For all $x, y \in G$, $L_s B_r^{(1)} \subset L_{x,y}$ and $[L_x, L_y]_{\mathscr{B}^{(1)}} \subset B_{x^{-1}y}^{(1)}$; $B_x^{(2)} L_y  \subset L_{x,y}$ and $_{\mathscr{B}^{(2)}}[L_x, L_y] \subset B_{xy^{-1}}^{(1)}$.

$(ii)$ For all $x,y \in G$ and $r \in L_x$ the function $L_x \times B_y^{(1)} \to L_{xy}, (r,a) \mapsto ra$ is bilinear and $L_y \to B_{x^{-1}y}, s \mapsto [r,s]_{\mathscr{B}^{(1)}}$ is linear; the function $ B_y^{(2)}  \times L_x \to L_{yx}, (b,r) \mapsto br$ is bilinear and $L_x \to B_{x y^{-1}}, s \mapsto \ _{\mathscr{B}^{(2)}}[s,r]$ is linear.

$(iii)$ For all $r,s \in L$ and $a \in B^{(1)}$ and $b \in B^{(2)}$, $[r,s]_{\mathscr{B}^{(1)}}^{\ast}=[s,r]_{\mathscr{B}^{(1)}}$, $[r,sa]_{\mathscr{B}^{(1)}}=[r,s ]_{\mathscr{B}^{(1)}} \ a$ and $_{\mathscr{B}^{(2)}}[r,s]^{\ast}=\ _{\mathscr{B}^{(2)}}[s,r]^{\ast}$, $_{\mathscr{B}^{(2)}}[br,s]=b \ _{\mathscr{B}^{(2)}}[r,s ]$; $[r,r ]_{\mathscr{B}^{(1)}} \geq 0$ in $B^{(1)}_e$ and $_{\mathscr{B}^{(2)}}[r,r ] \geq 0$ in $B^{(2)}_e$; $\|r\|^2=[r,r ]_{\mathscr{B}^{(1)}}=_{\mathscr{B}^{(2)}}[r,r ]$.

$(iv)$ For all $r, s$ and $t$ in $L$, $_{\mathscr{B}^{(2)}}[r,s]t=r[s,t]_{\mathscr{B}^{(1)}}$.

$(v)$ $B_e^{(1)}=\overline{{\rm{span}}}\{[r,s]_{\mathscr{B}^{(1)}}: r, s \in G\}$; $B_e^{(2)}=\overline{{\rm{span}}}\{_{\mathscr{B}^{(2)}}[r,s]: r, s \in G\}$.

If these conditions hold, we say that $\mathscr{B}^{(1)}$ and $\mathscr{B}^{(2)}$ are $equivalent$ $implemented$ $by$ $\mathcal{L}$. 

In the rest of this paper, $\mathbf{we \ assume \ that}$ $\mathscr{B}^{(1)}$ $\mathbf{and}$ $\mathscr{B}^{(2)}$ $\mathbf{are}$ $\mathbf{two \ equivalent}$  $\mathbf{Fell \ bundles \ over}$ $G$ $\mathbf{implemented \ by}$ $\mathcal{L}$.
\end{definition}

\begin{construction}
\rm We can define a right $\mathscr{L}(\mathscr{B}^{(1)})$-module and left $\mathscr{L}(\mathscr{B}^{(1)})$-module structure for $\mathscr{L}(\mathcal{L})$ by 
\begin{equation*}
(\mathfrak{f} \cdot f)(x)=\int_G \mathfrak{f}(y)f(y^{-1}x)dy; \ (g \cdot \mathfrak{f})(x)=\int_G g(y)\mathfrak{f}(y^{-1}x)dy  
\end{equation*}
$(\mathfrak{f} \in \mathscr{L}(\mathcal{L}), f \in \mathscr{L}(\mathscr{B}^{(1)}), g \in \mathscr{L}(\mathscr{B}^{(2)}))$. Furthermore, we define $[ \ , \ ]_{\mathscr{L}(\mathscr{B}^{(1)})}: \mathscr{L}(\mathcal{L}) \times \mathscr{L}(\mathcal{L}) \to \mathscr{L}(\mathscr{B}^{(1)})$ by
\begin{equation*}
[\mathfrak{f},\mathfrak{g}]_{\mathscr{L}(\mathscr{B}^{(1)})}(x)=\int_G [\mathfrak{f}(y), \mathfrak{g}(y^{-1}x)]_{\mathscr{B}^{(1)}}dy
\end{equation*}
$(\mathfrak{f}, \mathfrak{g} \in \mathscr{L}(\mathcal{L}))$, and $_{\mathscr{L}(\mathscr{B}^{(2)})}[ \ , \ ]: \mathscr{L}(\mathcal{L}) \times \mathscr{L}(\mathcal{L}) \to \mathscr{L}(\mathscr{B}^{(2)})$ by
\begin{equation*}
_{\mathscr{L}(\mathscr{B}^{(2)})}[\mathfrak{f},\mathfrak{g}](x)=\int_G {_{\mathscr{B}^{(2)}}[\mathfrak{f}(y), \mathfrak{g}(y^{-1}x)]}dy
\end{equation*}
$(\mathfrak{f}, \mathfrak{g}  \in \mathscr{L}(\mathcal{L}))$.

\end{construction}
\begin{remark}
\rm By \cite[Corollary 4.3]{vjnj}, $\langle \mathscr{L}(\mathcal{L}), _{\mathscr{L}(\mathscr{B}^{(2)})}[ \ , \ ],  [ \ , \ ]_{\mathscr{L}(\mathscr{B}^{(1)})}  \rangle$ is a $\mathscr{L}(\mathscr{B}^{(2)})$-$\mathscr{L}(\mathscr{B}^{(1)})$ imprimitivity system implementing the Morita equivalence between $C^{\ast}(\mathscr{B}^{(1)})$ and $C^{\ast}(\mathscr{B}^{(2)})$ (see Convention \ref{fdsickjldfsofxc}).

\end{remark}

\section{The Bridge Lemmas}

\subsection{Bridge Lemma: Path I}

\begin{lemma}\label{dsijkjldskjpoiadzx}
Let $A$ and $C$ be two $\ast$-algebras;  $\mathscr{L}_{C,A}=\langle L, [\ , \ ]_A \rangle$ is a right $A$-rigged left $C$-module. Let $B$ be a $\ast$-subalgebra of $A$ with an $A, B$ conditional expectation $p: A \to B$. The map $[ \ ,\  ]_B: L \times L \to B$ defined by
\begin{equation*}
[r, s]_B=p([r,s]_A) \ \ \ \ (r,s \in L)
\end{equation*}
making $L$ a right $B$-rigged left $C$-module, which we denote by $\mathscr{L}_{C,B}$. Assume that $S$ is a non-degenerate $\ast$-representation of $B$. If $S$ is inducible to $A$ via $p$ and ${\rm{Ind}}(S; p)$ is inducible to $C$ via $\mathscr{L}_{C,A}$, then $S$ is inducible to $C$ via $\mathscr{L}_{C,B}$. In this case, we have
\begin{equation*}
{\rm{Ind}}(S; \mathscr{L}_{C,B}) \simeq {\rm{Ind}}({\rm{Ind}}(S; p); \mathscr{L}_{C,A}).
\end{equation*}
\end{lemma}
\begin{proof}
Let define $\Gamma': L \times L \to \mathcal{O}(X(S))$ by  
\begin{equation*}
\langle s,r\rangle \mapsto S_{[r,s]_B} \ \ \ \ (r,s \in L).
\end{equation*}
 In order to prove that $S$ is inducible to $C$ via $\mathscr{L}_{C,B}$, our task is to prove that $S$ is $\mathscr{L}_{C,B}$-positive and that $\Gamma'$ is a $C$-bounded operator inner product (see Construction \ref{serijcdkjdiucjifdipier}).
 
By \cite[XI.4.5]{MR936629}, since $S$ is non-degenerate, the inequality 
\begin{equation}\label{sdkjkljadlicajlklca}
(S_{[t,t]_B}(S_a\xi), S_a\xi)=(S_{[t \cdot a, t \cdot a]_A}\xi, \xi)\geq 0 \ \ \ \ (t \in L; a \in B; \xi \in X(S)).
\end{equation}
implies that $S$ is $\mathscr{L}_{C,B}$-positive.  
 
 By our hypothesis and \cite[XI.5.6]{MR936629}, the map $\Gamma: (L \otimes A) \times (L \otimes A) \to \mathcal{O}(X(S))$ satisfying
\begin{equation}\begin{split}\label{adsjlklkjdkljesadsa}
\Gamma(s \otimes a, r \otimes b)&=S_{p(b^{\ast}[r,s]_A \ a)} 
\\&=S_{p([r \cdot b, s \cdot a]_A)}   \ \ \ \ (a, b \in A; r,s \in L)
\end{split}\end{equation}
is a $C$-bounded operator inner product on $L \otimes A$.
Then by (\ref{adsjlklkjdkljesadsa}) and 
\begin{equation}\begin{split}
\Gamma': ( r \cdot b, s \cdot a ) &\mapsto S_{[s \cdot a, r \cdot b]_B}
\\&=S_{p([ r \cdot b, s \cdot a]_A)},
 \end{split}\end{equation}
we conclude that $\Gamma'|((L \cdot A) \times (L \cdot A))$ is a $C$-bounded inner product on $L \cdot A$. But $S$ is non-degenerate, hence $\mathfrak{X}(\Gamma')$ is equal to $\mathfrak{X}(\Gamma'|((L \cdot A) \times (L \cdot A)))$ implemented by the unique unitary map satisfying
\begin{equation}\label{sdiiodksxoiw0qweioasdok}
(r \cdot a) \xi \mapsto r \otimes S_a(\xi) \ \ \ \ (r \in L; a \in B; \xi \in X(S)).
\end{equation}
Therefore, since $\Gamma'|((L \cdot A) \times (L \cdot A))$ is $C$-bounded, we can conclude that $\Gamma'$ is $C$-bounded. 

Therefore, $S$ is inducible to $C$ via $\mathscr{L}_{C,B}$; and it is easy to verify that the unitary map satisfying (\ref{sdiiodksxoiw0qweioasdok}) implementing the unitary equivalence of ${\rm{Ind}}(S; \mathscr{L}_{C,B})$ and ${\rm{Ind}}({\rm{Ind}}(S; p); \mathscr{L}_{C,A})$. Our proof is complete. 
\end{proof}

\begin{notation}
\rm Recall from Example \ref{eijaioadocareadas} that we have a conditional expectation $p^{(1)}: \mathscr{L}(\mathscr{B}^{(1)}) \to \mathscr{L}(\mathscr{B}_H^{(1)})$. Therefore, we can construct a right $\mathscr{L}(\mathscr{B}^{(1)}_H)$-rigged left $\mathscr{L}(\mathscr{B}^{(2)})$-module $\langle \mathscr{L}(\mathcal{L}), [ \ , \ ]' \rangle$, where $[ \ , \ ]': \mathscr{L}(\mathcal{L}) \times \mathscr{L}(\mathcal{L}) \to \mathscr{L}(\mathscr{B}^{(1)}_H)$ is defined by
\begin{equation*}
[\mathfrak{f}  ,  \mathfrak{g}]'=p^{(1)}([\mathfrak{f}  ,   \mathfrak{g}]_{\mathscr{L}(\mathscr{B}^{(1)})}) \ \ \ \ (\mathfrak{f}, \ \mathfrak{g} \in \mathscr{L}(\mathcal{L})),
\end{equation*}
We use the symbol $\mathcal{B}_{\mathscr{B}^{(1)}_H, \mathscr{B}^{(2)}}$ to denote $\langle \mathscr{L}(\mathcal{L}), [ \ , \ ]' \rangle$, which we call the $Bridge$ $from$ $\mathscr{B}^{(1)}_H$ to $\mathscr{B}^{(2)}$.
\end{notation}

$\mathbf{In \  the \  rest \ of \ this \ paper, if}$ $S$ $\mathbf{is \ a}$ $\ast$-$\mathbf{representation \ of \ a \ Fell \ bundle}$ $\mathscr{B}$, $\mathbf{then \ we \ denote}$ $\mathbf{its \ integrated \ form, \ which \ is \ a \ representation \ of }$ $\mathscr{L}(\mathscr{B})$, $\mathbf{by}$ $\widetilde{S}$. $\mathbf{But \ we \ shall \ write }$ $S$ $\mathbf{instead \ of}$ $\widetilde{S}$ $\mathbf{if \ no \ confusion \ can \ arise}$.

\begin{proposition}\label{dsiijodsoicserer}
$({\rm{Bridge \ Lemma \ I}})$ Let $S$ be a non-degenerate $\ast$-representation of $\mathscr{B}^{(1)}_H$. If $S$ is $\mathscr{B}^{(1)}$-positive, then $\widetilde{S}$ is inducible to $\mathscr{L}(\mathscr{B}^{(2)})$ via $\mathcal{B}_{\mathscr{B}^{(1)}_H, \mathscr{B}^{(2)}}$; and ${\rm{Ind}}(\widetilde{S}; \mathcal{B}_{\mathscr{B}^{(1)}_H, \mathscr{B}^{(2)}})$ is the integrated form of a $\ast$-representation of $\mathscr{B}^{(2)}$, which we denote by ${\rm{Ind}}(S; \mathcal{B}_{\mathscr{B}^{(1)}_H, \mathscr{B}^{(2)}})$. In this case, we say that ${\rm{Ind}}(S; \mathcal{B}_{\mathscr{B}^{(1)}_H, \mathscr{B}^{(2)}})$ is induced from $S$ via $\mathcal{B}_{\mathscr{B}^{(1)}_H, \mathscr{B}^{(2)}}$; and we have
\begin{equation}\label{sdioiosdkosddsweew}
{\rm{Ind}}({\rm{Ind}}_{\mathscr{B}^{(1)}_H \uparrow \mathscr{B}^{(1)}}(S); \langle \mathscr{L}(\mathcal{L}),  [ \ , \ ]_{\mathscr{L}(\mathscr{B}^{(2)})} \rangle)\simeq {\rm{Ind}}(S; \mathcal{B}_{\mathscr{B}^{(1)}_H, \mathscr{B}^{(2)}}).
\end{equation}
\end{proposition}
\begin{proof}
By Lemma \ref{dsijkjldskjpoiadzx} ${\rm{Ind}}(\widetilde{S}; \mathcal{B}_{\mathscr{B}^{(1)}_H, \mathscr{B}^{(2)}})$ exists; and by  Lemma \ref{dsijkjldskjpoiadzx} again we have
\begin{equation}\label{sjoiojdsjoids}
{\rm{Ind}}({\rm{Ind}}_{\mathscr{B}^{(1)}_H \uparrow \mathscr{B}^{(1)}}(S); \langle \mathscr{L}(\mathcal{L}),  [ \ , \ ]_{\mathscr{L}(\mathscr{B}^{(2)})} \rangle)^{\sim}|(\mathscr{L} (\mathscr{B}^{(2)}))  \simeq {\rm{Ind}} (\widetilde{S}; \mathcal{B}_{\mathscr{B}^{(1)}_H, \mathscr{B}^{(2)}}).
\end{equation}
Since the left side of (\ref{sjoiojdsjoids}) is the integrated form of a $\ast$-representation of $\mathscr{B}^{(2)}$, then so is ${\rm{Ind}}(\widetilde{S}; \mathcal{B}_{\mathscr{B}^{(1)}_H, \mathscr{B}^{(2)}})$. (\ref{sdioiosdkosddsweew}) is readily derived from (\ref{sjoiojdsjoids}).
\end{proof}

\subsection{Bridge Lemma: Path II}

\begin{definition}\label{sdiosdoisdopiwep9dkopsx}
\rm Let $\mathscr{Z}$ be a Hilbert bundle over $G/H$ with the bundle space $Z$ and fibers $\{Z_{\alpha}\}_{\alpha \in G/H}$.  Let$\theta$ is a $\ast$-representation of $\mathscr{B}$ on the Hilbert space $\mathscr{L}_2(\mathscr{Z})$. We assume that for any $a \in B$ there is a map $\theta'_a: Z \to Z$ satisfying

(i) $\theta_a(\mathsf{f})(\alpha)=\theta'_a(\mathsf{f}(\pi(a)^{-1}\alpha))$ for $a \in B$, $\mathsf{f} \in \mathscr{L}(\mathscr{Z})$ and $\alpha \in G/H$. In particular, $\theta'_a(Z_{\alpha}) \subset Z_{\pi(a)\alpha}$ $(a \in B;  \alpha \in G/H)$. 

(ii) For any $x \in G$ the closure of the linear span of $\{\theta'_a(\xi): a \in B_{xH}, \xi \in Z_H\}$ is $Z_{xH}$, and $\theta'_a|Z_{\alpha}$ is a bounded linear operator from $Z_{\alpha}$ into $Z_{\pi(a)\alpha}$.

(iii) $a \mapsto \theta'_a(\xi)$ is a continuous map from $B$ into $Z$ for any fixed $\xi \in Z$.

We say that ${\theta}$ is an $integrated$ $representation$ $of$ $\mathscr{B}$ on $\mathscr{Z}$; and we say that $\theta'$ is the $derivative$ of ${\theta}$.
\end{definition}

\begin{lemma}\label{dsiodfwiosdoierwc}
Suppose that ${\theta}$ is an integrated representation of $\mathscr{B}$ on the Hilbert bundle $\mathscr{Z}=\{Z_{\alpha}\}_{\alpha \in G/H}$ over $G/H$ with the derivative $\theta'$. By (i) of Definition \ref{dsiodfwiosdoierwc} $\theta'_a(Z_H) \subset Z_H$ for $a \in B_H$; so we can define a map on $\mathscr{B}_H$ into $\mathcal{O}(Z_H)$, which we denote by $^{Z_{H}}({\theta'}|B_H)$. Then $^{Z_{H}}({\theta'}|B_H)$ is a $\mathscr{B}^{(2)}$-positive $^{\ast}$-representation of $\mathscr{B}_H$. Let $\mathscr{Y}$ be the Hilbert bundle over $G/H$ induced from $^{Z_{H}}({\theta}|B_H)$; then there is a unique unitary map $\Psi: \mathscr{Y} \to \mathscr{Z}$ satisfying
\begin{equation*}
\Psi: \kappa_{\pi(a)}(a \otimes \xi) \mapsto \rho(\pi(a))^{-1}\theta'_a(\xi)  \ \ \ \ (a \in B; \ \xi \in Z_H);
\end{equation*}
and $\Psi$ implements the unitary equivalence of ${\theta}$ and ${\rm{Ind}}_{\mathscr{B}_{H} \uparrow \mathscr{B}}(^{Z_{H}}({\theta'}|B_H))$.
\end{lemma}
\begin{proof}
(ii) and (iii) of Definition \ref{dsiodfwiosdoierwc} imply that $\theta'|B_H$ is a $\ast$-representation of $\mathscr{B}_H$ on $Z_H$.  

For each $\alpha \in G/H$, let $( \ , \ )_{\alpha}$ be the map defined on $\mathscr{L}(\mathcal{L}_{\alpha})$ as in Construction \ref{seriodiofdsiofdsoierw}. Then for any $a, b \in B_{\alpha}$ and $\xi, \eta \in Z_H$, I have
\begin{equation}\begin{split}
(a \otimes \xi, b \otimes \eta)_{\alpha}&=(\theta'_{b^{\ast}a}(\xi), \eta)_{Z_H}
\\&=(\theta'_a(\xi), \theta'_b(\eta))_{Z_{\alpha}}.
\end{split}\end{equation}
Then it is easy to see that $( \ , \ )_{\alpha}$ is positive; and so $^{Z_{H}}({\theta'}|B_H)$ is $\mathscr{B}^{(2)}$-positive. The other parts may be verified by trivial computations.
\end{proof}

$\mathbf{In \ the \ rest \ of \ this \ section, \ we\ fix \ a \ non}$ $\mathbf{degenerate}$ $\mathscr{B}^{(1)}$-$\mathbf{positive}$ $\ast$-$\mathbf{repre}$-$\mathbf{esentation}$ $S$ $\mathbf{of}$ $\mathscr{B}^{(1)}_H$, $\mathbf{and \ write}$ $X$ $\mathbf{for}$ $X(S)$.

\begin{construction}\label{sdijodsojidiodakadqwe}
\rm  Our first goal is to construct from $S$ a Hilbert bundle $\mathscr{Z}$ over
$G/H$. To this end the first step is to construct a Hilbert space $Z_{\alpha}$ for each coset
$\alpha \in G/H$. Let $x \in \alpha \in G/H$. We form the algebraic tensor product $\mathscr{L}({\mathcal{L}_{\alpha}}) \otimes X$, and introduce into it the conjugate-bilinear form $(\ , \ )_{\alpha}$ given by
\begin{equation}\label{sdjioijsdkscweewew}
(\mathfrak{r} \otimes \xi, \mathfrak{s} \otimes \eta)_{\alpha}=\int_H \int_H (\rho(xh)\rho(xk))^{-1/2}(S_ { [\mathfrak{s}(xk), \mathfrak{r}(xh))]_{\mathscr{B}^{(1)}}}\xi, \eta)dvh dvk
\end{equation}
$(\mathfrak{r}, \mathfrak{s} \in \mathscr{L}(\mathcal{L}_{\alpha}); \xi, \eta \in X)$. 


\begin{lma}
$( \ , \ )_{\alpha}$ is positive.
\end{lma}
\begin{proof}
For any $\mathfrak{f}, \mathfrak{g} \in \mathscr{L}(\mathcal{L})$, we have 
\begin{equation}\begin{split}\label{asdijoiojdwjadswe}
&\ \ \ \  (\mathfrak{f} \widetilde{\otimes}\xi, \mathfrak{g}\widetilde{\otimes}\eta)_{X({\rm{Ind}}(S; \mathcal{B} ))}
\\&=(S_{[\mathfrak{g},\mathfrak{f}]_{\mathscr{L}(\mathscr{B}^{(1)})}}\xi, \eta)_X
\\&=\int_{G/H}d\rho^{\sharp}(xH)\int_H\int_H(\rho(xh))\rho(xk))^{-1/2}(S_{[\mathfrak{g}(xh), \mathfrak{f}(xk)]_{\mathscr{B}^{(1)}}}\xi,\eta )dvhdvk. 
\end{split}\end{equation}
Let $\mathfrak{r}, \mathfrak{s} \in \mathscr{L}(\mathcal{L}_{\alpha})$. By \cite[II.14.8]{MR936628}, there are $\mathfrak{f}, \mathfrak{g} \in \mathscr{L}(\mathcal{L})$ such that $\mathfrak{f}|\alpha=\mathfrak{r}$ and $\mathfrak{g}|\alpha=\mathfrak{s}$. Let $\phi_i \in \mathscr{L}(G/H)$ be an approximate unit of $G/H$ around $\alpha$, by (\ref{asdijoiojdwjadswe}) we have
\begin{equation}\label{sdijoioejodsaadqwe}
(\mathfrak{r} \otimes \xi, \mathfrak{s}\otimes \eta)_{\alpha}
={\rm{lim}}_{i \to \infty}((\phi_i \mathfrak{f}) \widetilde{\otimes} \xi, (\phi_i \mathfrak{g}) \widetilde{\otimes} \eta)_{X({\rm{Ind}}(S; \mathcal{B} ))}.
\end{equation}
Combining (\ref{sdijoioejodsaadqwe}), (\ref{asdijoiojdwjadswe}) and Proposition \ref{dsiijodsoicserer}, it is easy to see that $( \ , \ )_{\alpha}$ is positive.
\end{proof}
\begin{nota}
\rm For each $\alpha \in G/H$, let $Z_{\alpha}$ be the Hilbert space resulting by quotient and completion from $\mathscr{L}(\mathscr{Z}_{\alpha}) \otimes X$; and we use the symbol $\iota_{\alpha}$ to denote the quotient map from $\mathscr{L}(\mathscr{Z}_{\alpha}) \otimes X$ into $Z_\alpha$. Let $Z$ be the disjoint union of the $Z_{\alpha} (\alpha \in G/H)$.
\end{nota}
\begin{prop}
There is a unique topology for Z making $Z=\langle Z, \{Z_{\alpha}\}\rangle$ a
Hilbert bundle over $G/H$, which we denote by $\mathscr{Z}$, such that for each $\mathfrak{f}$ in $\mathscr{L}(\mathcal{L})$ and  $\xi$ in $X$ the cross-section
\begin{equation*}
\alpha \mapsto \iota_{\alpha}((\mathfrak{f}|\alpha)\otimes \xi)
\end{equation*}
of $\mathscr{Z}$ is continuous.
\end{prop}
\end{construction}


\begin{proposition}\label{dooadslksdpoqwepoadslksxaxz}
The induced representation ${\rm{Ind}}(S; \mathcal{B}_{\mathscr{B}^{(1)}_{H}, \mathscr{B}^{(2)}})$ $($acting on $\mathscr{L}(\mathcal{L}) \otimes_S X(S)$$)$ exists; and there is a unique unitary map $E: \mathscr{L}(\mathcal{L}) \otimes_S X(S) \to \mathscr{L}_2(\mathscr{Z})$ satisfying
\begin{equation*}
\mathfrak{f} \widetilde{\otimes}\xi  \mapsto (\alpha \mapsto \iota_{\alpha}((\mathfrak{f}|\alpha) \otimes \xi) \ \ \ \ (\mathfrak{f} \in \mathscr{L}(\mathscr{Z})).
\end{equation*}
In particular, we can define a $\ast$-representation $T$ of $\mathscr{B}^{(2)}$ on $\mathscr{L}_2(\mathscr{Z})$ satisfying 
\begin{equation}\label{sdijoiasioqweoiasdkasdzx}
T_a(\alpha \mapsto \iota_{\alpha}((\mathfrak{f}|\alpha) \otimes \xi))=\alpha  \mapsto \iota_{\alpha}((a\mathfrak{f})|\alpha \otimes \xi);  
\end{equation}
$T$ and ${\rm{Ind}}(S; \mathcal{B}_{\mathscr{B}^{(1)}_{H}, \mathscr{B}^{(2)}})$ are unitarily equivalent implemented by $E$.
\end{proposition}
\begin{proof}
Since $S$ is $\mathscr{B}^{(1)}$-positive, then by Proposition \ref{dsiijodsoicserer} ${\rm{Ind}}(S; \mathcal{B}_{\mathscr{B}^{(1)}_{H}, \mathscr{B}^{(2)}})$ exists. 
By (\ref{asdijoiojdwjadswe}) and the construction of $\mathscr{Z}$, for any $\mathfrak{f}, \mathfrak{g} \in \mathscr{L}(\mathcal{L})$ and $\xi, \eta \in X$ we have
\begin{equation*}
(\mathfrak{f} \widetilde{\otimes}\xi, \mathfrak{g} \widetilde{\otimes} \eta)_{X({\rm{Ind}}(S; \mathcal{B} ))}=(\alpha  \mapsto \iota_{\alpha}(\mathfrak{f}|\alpha \otimes \xi), \alpha  \mapsto \iota_{\alpha}(\mathfrak{g}|\alpha \otimes \eta))_{\mathscr{L}_2(\mathscr{Z})}.
\end{equation*}
Therefore the linear map $\mathfrak{f} \widetilde{\otimes}\xi  \mapsto (\alpha \mapsto \iota_{\alpha}((\mathfrak{f}|\alpha) \otimes \xi)$ can be extended to a unitary map $E$ from $\mathscr{L}(\mathcal{L})\otimes_{S}X(S)$ into $\mathscr{L}_2(\mathscr{Z})$; and by \cite[II.14.6]{MR936628} this map is surjective; and it is easy to verify that 
\begin{equation*}
E{\rm{Ind}}(S; \mathcal{B}_{\mathscr{B}^{(1)}_{H}, \mathscr{B}^{(2)}})_aE^{\ast}(\alpha \mapsto \iota_{\alpha}((\mathfrak{f}|\alpha) \otimes \xi))=\alpha \mapsto \iota_{\alpha}(((a\mathfrak{f})|\alpha) \otimes \xi) \ \ \ \ \ (a \in B^{(2)}).
\end{equation*}
So we can define $T_a=E{\rm{Ind}}(S; \mathcal{B}_{\mathscr{B}^{(1)}_{H}, \mathscr{B}^{(2)}})_aE^{\ast}$ for each $a \in B^{(2)}$. Then $a \mapsto T_a$ is a $\ast$-representation of $\mathscr{B}^{(2)}$ satisfying (\ref{sdijoiasioqweoiasdkasdzx}); and $T$ and ${\rm{Ind}}(S; \mathcal{B}_{\mathscr{B}^{(1)}_{H}, \mathscr{B}^{(2)}})$ are unitarily equivalent implemented by $E$.
\end{proof}

\begin{definition}\label{sdjidsijodfjiocxpoierdas}
\rm We say that the Hilbert bundle constructed in Construction \ref{sdijodsojidiodakadqwe} the $Hilbert$ $bundle$ $induced$ $from$ $S$ $via$ $\mathcal{B}_{\mathscr{B}^{(1)}_{H}, \mathscr{B}^{(2)}}$. 
\end{definition}

\begin{notation}\label{dejdoijdsaqwewqe}
\rm Let $F$ be a norm space; $P \subset F$. We use the symbol $\lceil P \rfloor$ to denote the closure of the linear span of $P$.
\end{notation}

\begin{definition}
\rm We say that $\mathcal{L}$ is $\mathbi{$\mathscr{B}^{(2)}$}$-$\mathbi{full}$ $\mathbi{over}$ $\mathbi{G/H}$ if for each $x \in G$ and $\alpha \in G/H$ we have
\begin{equation*}
\ulcorner\{ar: r \in L_{y}, y \in \alpha, a \in B_x\} \lrcorner=\mathcal{L}_{x\alpha}.
\end{equation*}
\end{definition}



By routine computations we can prove the following useful lemma:
\begin{lemma}\label{erwiooisdksqooewew}
The map $\Psi_H: \mathscr{L}(\mathcal{L}_H) \otimes X(S) \to \mathscr{L}(\mathcal{L}_H) \otimes_S X(S)$ defined by
\begin{equation*}
\Psi_H(\mathfrak{r} \otimes \xi)=\rho(e)^{-1}\mathfrak{r}' \widetilde{\otimes} \xi \ \ \ \ (\mathfrak{r} \in \mathscr{L}(\mathcal{L}_H); \ \xi \in X(S)), 
\end{equation*}
where $\mathfrak{r}' \in \mathscr{L}(\mathcal{L}_H)$ is defined by $\mathfrak{r}'(h)=\Delta(h)^{1/2}\delta(h)^{-1/2}\mathfrak{r}(h)$, can be extended to a unitary map from $Z_H$ onto $\mathscr{L}(\mathcal{L}_H) \otimes _S X(S)$. We denote this unitary map still by $\Psi_H$.
\end{lemma}

\begin{theorem}\label{sdoiewiosdsdwesxdsdc}
$(${\rm{Bridge Lemma II}}$)$ If $\mathcal{L}$ is $\mathscr{B}^{(2)}$-full over $G/H$, then the $\ast$-$representat$-ion $T$ defined in Construction \ref{dooadslksdpoqwepoadslksxaxz} on $\mathscr{L}_2(\mathscr{Z})$ is an integrated representation of $\mathscr{B}^{(2)}$; and ${\rm{Ind}}(S; \mathscr{L}(\mathcal{L}_{H}))$ exists and is $\mathscr{B}^{(2)}$-positive; and we have
\begin{equation}\label{sowqoiszswwxcf}
{\rm{Ind}}_{\mathscr{B}_H^{(2)} \uparrow \mathscr{B}^{(2)}}({\rm{Ind}}(S; \mathscr{L}(\mathcal{L}_{H})))\simeq{\rm{Ind}}(S; \mathcal{B}_{\mathscr{B}_H^{(1)}, \mathscr{B}^{(1)}}).
\end{equation}
\end{theorem}
\begin{proof}
For each $a \in B^{(2)}$ and $\alpha \in G/H$, we define $t^{\alpha}_a: \mathscr{L}(\mathcal{L}_{\alpha}) \otimes X \to \mathscr{L}(\mathcal{L}_{\pi(a)\alpha}) \otimes X$ by
\begin{equation*}
t^{\alpha}_a(\mathfrak{r} \otimes \xi)=a\mathfrak{r} \otimes \xi \ \ \ \ (\mathfrak{r} \in \mathscr{L}(\mathcal{L}_{\alpha}), \xi \in X),
\end{equation*}
where $a\mathfrak{r} \in \mathscr{L}(\mathcal{L}_{\pi(a)\alpha}) $ is defined by
\begin{equation*}
(a\mathfrak{r})(\pi(a)x)=a\mathfrak{r}(x) \ \ \ \ (x \in \alpha).
\end{equation*}
Let $\psi_w: G/H \to \mathbb{R}$ be a net of approximation unit of $G/H$ around $\alpha$. For $\sum_{i=1}^na\mathfrak{r}_i \otimes \xi_i \in \mathscr{L}(\mathcal{L}_{\alpha}) \otimes X$, let $\mathfrak{f}_i \in \mathscr{L}(\mathcal{L})$ such that $\mathfrak{f}_i|\alpha=\mathfrak{r}_i$; we have
\begin{equation*}\begin{split}
&\ \ \ \ \|\sum_{i=1}^na\mathfrak{r}_i \otimes \xi_i\|_{Y_{\pi(a)\alpha}}^2
\\&=\sum_{i,j=1}^n\int_H \int_H (\rho(\pi(a)xh)\rho(\pi(a)xk))^{-1/2}(S_{[(a\mathfrak{r}_i)(\pi(a)xk), (a\mathfrak{r}_j)(\pi(a)xh))] }\xi_j, \xi_i)dvh dvk
\\&=\sum_{i,j=1}^n\int_H \int_H (\rho(\pi(a)xh)\rho(\pi(a)xk))^{-1/2}(S_{[(a \cdot \mathfrak{f}_i)(\pi(a)xk), (a \cdot \mathfrak{f}_j)(\pi(a)xh))] }\xi_j, \xi_i)dvh dvk
\\&= {\rm{lim}}_{w \to \infty}( \sum_{i,j=1}^n\int_{G/H}d\rho^{\sharp}(\pi(a)xH)  
 \\& \ \ \ \ \ \ \ \ \ \int_H \int_H (\rho(\pi(a)xh)\rho(\pi(a)xk))^{-1/2}(S_{[(a \cdot \psi_w\mathfrak{f}_i)(\pi(a)xk), (a \cdot \psi_w\mathfrak{f}_j)(\pi(a)xh))] }\xi_j, \xi_i)
 \\& \ \ \ \ \ \ \ \ \ \ \ \ \ \ \ \ \ \ \ \ dvhdvk) 
\\&= {\rm{lim}}_{w \to \infty}\|T_a(\sum_{i=1}^n\psi_w\mathfrak{f}_i \otimes \xi_i)\|^2_{Y_{\alpha}}
\\& \leq {\rm{lim}}_{w \to \infty}\|\sum_{i=1}^n\psi_w\mathfrak{f}_i \otimes \xi_i\|^2_{Y_{\alpha}}
\\&={\rm{lim}}_{w \to \infty}( \sum_{i,j=1}^n\int_{G/H}d\rho^{\sharp}(xH)  
\\& \ \ \ \ \ \ \ \ \ \ \ \int_H \int_H (\rho(xh)\rho(xk))^{-1/2}(S_{[(\psi_w\mathfrak{f}_i)(xk), (\psi_w\mathfrak{f}_j)(xh))] }\xi_j, \xi_i)dvh dvk)
\\&=\int_H \int_H (\rho(xh)\rho(xk))^{-1/2}(S_{[\mathfrak{f}_i(xk), \mathfrak{f}_j(xh))] }\xi_j, \xi_i)dvh dvk
\\&=\|\sum_{i=1}^n\mathfrak{r}_i \otimes \xi_i\|^2_{Y_{\alpha}}.
\end{split}\end{equation*}
Therefore, $t^{\alpha}_a$ can be extended to a bounded linear operator from $Z_{\alpha}$ into $Z_{\pi(a)\alpha}$. Since $\mathcal{L}$ is $\mathscr{B}^{(2)}$-full over $G/H$, it is easy to prove that the linear span of $\{t^{\alpha}_a(\mathsf{f}): a \in B^{(2)}_{xH}, \ \mathsf{f}\in Z_{\alpha}\}$ is dense in $Z_{x\alpha}$ by the aid of \cite[II.14.1]{MR936628}. Let $t_a: Z \to Z$ be the unique map satisfying $t_a|Z_{\alpha}=t_a^{\alpha}$ for any $\alpha \in G/H$. Then $t$ satisfies (ii) of Definition \ref{dsiodfwiosdoierwc}. For each $\mathfrak{f} \in \mathscr{L}(\mathcal{L})$, let $\mathfrak{f}^{\xi}$ be the cross-section of $\mathscr{L}(\mathscr{Z})$ defined by $\alpha \mapsto \mathfrak{f}|\alpha \otimes \xi$; we have
\begin{equation}\label{edjdsajadfiucxkjadiodiosad}
T_a(\mathfrak{f}^{\xi})(\alpha)=t_a(\mathfrak{f}^{\xi}(\pi(a)^{-1}\alpha)) \ \ \ \ \ (a \in B^{(2)}; \ \alpha \in G/H);
\end{equation}
but since the linear span of $\{\mathfrak{f}^{\xi}: \mathfrak{f} \in \mathscr{L}(\mathcal{L}), \xi \in X\}$ is dense in the induct limit topology of $\mathscr{L}(\mathscr{Z})$, hence for any $\mathsf{f} \in \mathscr{L}(\mathscr{Z})$ we have 
\begin{equation*}
T_a(\mathsf{f})(\alpha)=t_a(\mathsf{f}(\pi(a)^{-1}\alpha)) \ \ \ \ (a \in B^{(2)}; \ \alpha \in G/H).
\end{equation*}
So (i) of Definition \ref{dsiodfwiosdoierwc} is fulfilled. (iii) of Definition \ref{dsiodfwiosdoierwc} is an easy consequence of the continuity of the action of $\mathscr{B}^{(2)}$ on $\mathcal{L}$ and the definition of the topology of the Hilbert bundle $\mathscr{Z}$. Therefore, $T$ is an integrated representation of $\mathscr{B}^{(2)}$ on $\mathscr{Z}$ with the derivative $t$.

For each $a \in \mathscr{B}_H^{(2)}$ let $S'_a$ be the bounded operator $t_a|Z_H: Z_H \to Z_H$; then it is easy to verify that $a \mapsto S'_a$ is a $\ast$-representation of $\mathscr{B}_H^{(2)}$ on the Hilbert space $Z_H$. Let $\Psi_H$ be the unitary map defined as in Lemma \ref{erwiooisdksqooewew}, we have
\begin{equation*}
\Psi_H S'_a \Psi_H^{-1}(\mathfrak{r} \widetilde{\otimes}\xi)=a\mathfrak{r} \widetilde{\otimes} \xi \ \ \ \ (a \in B_H^{(2)}; \ \mathfrak{r} \in \mathscr{L}(\mathcal{L}_H); \ \xi \in X(S)).
\end{equation*}  
Therefore, ${\rm{Ind}}(S; \mathscr{L}(\mathcal{L}_H))$ exists and is unitarily equivalent to $S'$; and by Lemma \ref{dsiodfwiosdoierwc} ${\rm{Ind}}(S; \mathscr{L}(\mathcal{L}_H))$ is $\mathscr{B}^{(2)}$-positive. By Proposition \ref{dooadslksdpoqwepoadslksxaxz} and Lemma \ref{dsiodfwiosdoierwc} we have (\ref{sowqoiszswwxcf}). Our proof is complete.
\end{proof}

\subsection{Bridge Lemma: Path III}
The following theorem is our main result of this section, which is a combination of Proposition \ref{dsiijodsoicserer} and Theorem \ref{sdoiewiosdsdwesxdsdc}.

\begin{theorem}\label{weioweiordeddsaf}
Suppose $\mathcal{L}$ is $\mathscr{B}^{(2)}$-full over $G/H$. If $S$ is a $\mathscr{B}^{(1)}$-positive non-degenerate $\ast$-representation of $\mathscr{B}_H^{(1)}$, then ${\rm{Ind}}(S; \mathscr{L}(\mathcal{L}_H))$ exists and is $\mathscr{B}^{(2)}$-positive; and we have
\begin{equation*}
{\rm{Ind}}({\rm{Ind}}_{\mathscr{B}_H^{(1)} \uparrow \mathscr{B}^{(1)}}(S); \mathscr{L}(\mathcal{L})) \simeq {\rm{Ind}}_{\mathscr{B}_H^{(2)} \uparrow \mathscr{B}^{(2)}}({\rm{Ind}}(S; \mathscr{L}(\mathcal{L}_H))).
\end{equation*}
\end{theorem}


\section{Application: An Investigation on the Fell-Doran Imprimitivity Theorem}


\subsection{A Remark}

In order to apply Theorem \ref{weioweiordeddsaf}, we need to know the conditions which imply that $\mathcal{L}$ is $\mathscr{B}^{(2)}$-full over $G/H$.

\begin{definition}
\rm We say that $\mathscr{B}$ is $H$-$saturated$ if for each $x \in G$ we have
\begin{equation*}
B_e \subset \lceil B_{xH}B^{\ast}_{xH} \rfloor.
\end{equation*}

\end{definition}

\begin{lemma}\label{dsiodiojcewqs}
For any $x \in G$, we have $\lceil B^{(2)}_e L_{xH} \rfloor=L_{xH}$.
\end{lemma}
\begin{proof}
In the following we use the symbol $[ \ , \ ]$ instead of $_{\mathscr{B}^{(2)}}[ \ , \ ]$. Let $S$ be a faithful $\ast$-representation of $B_e^{(2)}$. Let $\{a_i\}_{i \in I}$ be an approximation identity of $B_e^{(2)}$; then for any $r \in L_{xH}$  we have
\begin{equation*}\begin{split}
\|a_ir-r\|^2&=\|[a_ir-r, a_ir-r]\|
\\&=S_{[a_ir-r, a_ir-r]}
\\&=S_{a_i^2[r,r]-a_i^{\ast}[r,r]-a_i[r,r]+[r,r]};
\end{split}\end{equation*}
hence if $i \to \infty$, then $\|a_ir-r\| \to 0$. Therefore the linear span of $\{ar: a \in B_e^{(2)}, r \in L_{xh}\}$ is dense in $L_{xh}$ for any $h \in H$. Our proof is complete.
\end{proof}

\begin{lemma}\label{ewioiowekjdsoiwerdcs}
If $\mathscr{B}^{(2)}$ is $H$-saturated, then $\mathcal{L}$ is $\mathscr{B}^{(2)}$-full over $G/H$.
\end{lemma}
\begin{proof}
It is sufficient to prove that $\lceil B^{(2)}_{xH}L_{H} \rfloor=L_{xH}$ for any $x \in G$. We have
\begin{equation*}\begin{split}
L_{xH}=\lceil B^{(2)}_e L_{xH} \rfloor \subset \lceil B^{(2)}_{xH}B^{(2)\ast}_{xH}L_{xH} \rfloor=\lceil B^{(2)}_{xH}(B^{(2)\ast}_{xH}L_{xH})  \rfloor \subset \lceil B^{(2)}_{xH}L_{H}\rfloor.
\end{split}\end{equation*}
On the other hand, it is easy to see that $\lceil B^{(2)}_{xH}L_{H}\rfloor \subset L_{xH}$. Our proof is complete.
\end{proof}

\begin{lemma}\label{weisdkjdsijewripdx}
If $\mathscr{B}^{(1)}$ and $\mathscr{B}^{(2)}$ are strongly equivalent implemented by $\mathcal{L}$, then $\mathcal{L}$ is $\mathscr{B}^{(2)}$-full over G/H.
\end{lemma}
\begin{proof}
It is sufficient to prove that $\lceil B^{(2)}_{xH}L_{H} \rfloor=L_{xH}$ for any $x \in G$. By Lemma \ref{dsiodiojcewqs} we have
\begin{equation}\label{dsisdklcd}
\lceil B^{(2)}_{xH}L_H \rfloor  \supset \  \lceil_{\mathscr{B}^{(2)}}[L_{xH}, L_H] L_H \rfloor=\lceil L_{xH}  [L_H, L_H]_{\mathscr{B}^{(1)}}\rfloor;
\end{equation}
but since $\mathscr{B}^{(1)}$ and $\mathscr{B}^{(2)}$ are strongly equivalent, hence we have
\begin{equation}\label{sdidklxwe}
\lceil[L_H, L_H]_{\mathscr{B}^{(1)} }\rfloor \supset \lceil B^{(1)}B_e^{(1)\ast} \rfloor \supset B_e^{(1)}.
\end{equation}
Combining (\ref{dsisdklcd}) and (\ref{sdidklxwe})  we conclude that
\begin{equation}\label{sdaiasdkasw}
\lceil B^{(2)}_{xH}L_H \rfloor \supset L_{xH}.
\end{equation}
(\ref{sdaiasdkasw}) and $\lceil B^{(2)}_{xH}L_H \rfloor \subset L_{xH}$ completes our proof.
\end{proof}

\subsection{Application I: The Positivity of Representation}

\begin{lemma}\label{eijdsksdioewkzxaw}
Suppose $\mathscr{B}^{(1)}$ and $\mathscr{B}^{(2)}$ are strongly equivalent implemented by $\mathcal{L}$. If $S$ is a non-degenerate $\ast$-representation of $\mathscr{B}_H^{(1)}$ which is $\mathscr{B}^{(1)}$-positive, then ${\rm{Ind}}(S; \mathcal{L}_H)$ is $\mathscr{B}^{(2)}$-positive.
\end{lemma}
\begin{proof}
This is an easy consequence of Lemma \ref{weisdkjdsijewripdx} and Theorem \ref{weioweiordeddsaf}.
\end{proof}

\begin{theorem}\label{erwiosdiojedxsqwde}
Let $\mathscr{B}$ be a Fell bundle over $G$. Any non-degenerate $\ast$-representation of $\mathscr{B}_H$ is $\mathscr{B}$-positive.
\end{theorem}
\begin{proof}
If $\mathscr{B}$ is a semi-direct product bundle, then apply \cite[XI.8.9]{MR936629} it is easy to verify that any $\ast$-representation of $\mathscr{B}_H$ is $\mathscr{B}$-positive.

Now we drop the assumption that $\mathscr{B}$ is a semi-direct product bundle. Let $\mathscr{B}=\mathscr{B}^{(2)}$; and let $\mathscr{B}^{(1)}$ be a semi-direct product bundle which  is strongly equivalent to $\mathscr{B}=\mathscr{B}^{(2)}$ implemented by $\mathcal{L}$. Then $C^{\ast}(\mathscr{B}^{(1)}_H)$ and $C^{\ast}(\mathscr{B}_H^{(2)})$ are Morita equivalent implemented by $\mathscr{L}(\mathcal{L}_H)$. Hence there is a non-degenerate $\ast$-representation $S'$ of $\mathscr{B}_H^{(1)}$ such that ${\rm{Ind}}(S'; \mathscr{L}(\mathcal{L}_H))=S$. We have known that $S'$ is $\mathscr{B}^{(1)}$-positive; by Lemma \ref{eijdsksdioewkzxaw} $S$ is $\mathscr{B}^{(2)}$-positive. Our proof is complete.
\end{proof}

\subsection{Application II: Fell-Doran's Imprimitivity Theorem and Conjugation of Representations}

\subsubsection{Background}
We review the notions from \cite{MR936629} which we need in our sequential discussions. 
\begin{definition}\label{ewiodsjweoijddexs}
\rm $A$ $system$ $of$ $imprimitivity$ ($for$ $\mathscr{B}$ $over$ $G/H$) is a pair $\langle T, P \rangle$, where $T$ is a non-degenerate $\ast$-representation of $\mathscr{B}$ and $P$ a regular projection-valued measure of $G/H$ on $\mathcal{O}(X(T))$ satisfying
\begin{equation*}
T_b(P(W))=P(\pi(b)W)T_b  
\end{equation*} 
for all $b \in B$ and Borel subsets $W$ of $G/H$.
\end{definition}

\begin{construction}\label{weijodsiodwioqewxsaddwf}
\rm  For each $x \in G$ we let $D_x=\mathscr{C}_0(G/H, B_x)$ (the space of continuous functions from $G/H$ into $B_x$ vanishing at the infinity point). Let $D$ be the disjoint union of $\{D_x\}_{x \in G}$. By \cite[\S VIII.7]{MR936629} $D$ is a Fell bundle over $G$ with the multiplication and involution defined by
\begin{equation*}\begin{split}
&\psi\phi(m)=\psi(m)\phi(x^{-1}m),
\psi^{\ast}(m)=(\psi(xm))^{\ast}
\end{split}\end{equation*}
($x, y \in G; \ \psi \in D_x,\ \phi \in D_y; m \in G/H$). We call this Fell bundle, which we denote by $\mathscr{D}$, the $G, G/H$ $transformation$ $bundle$ $derived$ $from$ $\mathscr{B}$. 

 For $f \in \mathscr{L}(\mathscr{B})$ and $r \in \mathscr{C}_0(G/H)$, we use the symbol $fr$ to denote the the map $fr: G \to D$ defined by $fr(x)(m)=f(x)r(m)$ ($x \in G, \ m \in G/H$), and the linear span of $\{fr: x \in G, f \in \mathscr{L}(\mathscr{B}), \ r \in \mathscr{C}_0(G/H, B_x)\}$ by $\Gamma_{\mathscr{D}}$. By \cite[VIII.18.17]{MR936629}, $\Gamma_{\mathscr{D}}$ is a $\ast$-algebra dense in $\mathscr{L}(\mathscr{D})$ in the inductive limit topology.

By \cite[\S VIII.18]{MR936629} any non-degenerate $\ast$-representation of $\mathscr{D}$ can be identified with a unique system of imprimitivity $\langle T, P \rangle$; and by \cite[VIII.18.17]{MR936629}, under this identification we have 
\begin{equation}\label{sdisdicewewqdsa}
\langle T, P \rangle (fr)=P(r)T(f) \ \ \ \ (r \in \mathscr{C}_0(G/H); \ f \in \mathscr{L}(\mathscr{B})).
\end{equation}
\end{construction}

\begin{construction}\label{eoisdoisdfiocxoierer}
\rm Let $_{\mathscr{L}(\mathscr{D})}[ \ , \ ]: \mathscr{L}(\mathscr{B}) \times \mathscr{L}(\mathscr{B}) \to \mathscr{L}(\mathscr{D})$ be defined by
\begin{equation*}
_{\mathscr{L}(\mathscr{D})}[f,g](x, yH)=\int_H f(yh)g^{\ast}(h^{-1}y^{-1}x)dvh \ \ \ \ (x,y \in G).
\end{equation*}
Recall from Construction \ref{rseidierdcdewer} the definition of $[ \ , \ ]_{\mathscr{L}(\mathscr{B}_H)}$. By \cite[\S XI.14]{MR936629} we know that $\langle \mathscr{L}(\mathscr{B}), _{\mathscr{L}(\mathscr{D})}[\ , \ ], [ \ , \ ]_{\mathscr{L}(\mathscr{B}_H)}) \rangle$ is a $\mathscr{L}(\mathscr{D})$, $\mathscr{L}(\mathscr{B}_H)$ imprimitivity bimodule. Let $F_\mathscr{D}=\{_{\mathscr{L}(\mathscr{D})}[f,g]: f, g \in \mathscr{L}(\mathscr{B})\}$. 

Let $S$ be a non-degenerate $\ast$-representation of $\mathscr{B}_H$; and let $\mathscr{Y}$ be the Hilbert bundle over $G/H$ induced from $S$. For each Borel subset $W \subset G/H$, we define a projection $P(W)$ by
\begin{equation*}
P(W)f(x)=\chi_W (m)f(m) \ \ \ \ (f \in \mathscr{L}_2(\mathscr{Y}); \ m \in G/H)
\end{equation*}
for all $f \in \mathscr{L}(\mathscr{B})$ and $r \in \mathscr{C}_0(G/H)$.
Then $\langle {\rm{Ind}}_{\mathscr{B}_H \uparrow \mathscr{B}}(S), P \rangle$ is a system of imprimitivity; and we have
\begin{equation*}
\langle {\rm{Ind}}_{\mathscr{B}_H \uparrow \mathscr{B}}(S), P \rangle \simeq {\rm{Ind}}(S; \langle \mathscr{L}(\mathscr{B}),  [ \ , \ ]_{\mathscr{L}(\mathscr{B}_H)}, \ _{\mathscr{L}(\mathscr{D})}[ \ , \ ] \rangle).
\end{equation*}
\begin{definition}\label{adkjldsalidwxdw}
\rm We say that $\langle{\rm{Ind}}_{\mathscr{B}_H \uparrow \mathscr{B}}(S), P \rangle$ $is$ $induced$ $from$ $S$.
\end{definition}
\end{construction}

\begin{construction}
\rm Let $S$ be a non-degenerate $\ast$-representation of $\mathscr{B}_H$. Let $\mathscr{Y}$ be the Hilbert bundle over $G/H$ induced from $S$ (See Construction \ref{seriodiofdsiofdsoierw}). Let $x \in G$. For each $c \in B_{xHx^{-1}}$, we define $S'_c: Y_{xH} \to Y_{xH}$ by  
\begin{equation*}
S'_c(_x\xi)=\delta(h)^{(1/2)}\Delta^{(-1/2)}\tau_c(_x\xi)  \ \ \ \ (_x\xi \in Y_{xH}).
\end{equation*}
Then $S'_c$ is a bounded operator on $Y_{xH}$, and $c \mapsto S'_c$ is a $\ast$-representation of $\mathscr{B}_{xHx^{-1}}$ which we denote by $^{x}S$; and we say that $^{x}S$ is the $x$-$conjugate$ of $S$.

By \cite[XI.16.11]{MR936629}, for each $x \in G$ there is a $\mathscr{L}(\mathscr{B}_H)$, $\mathscr{L}(\mathscr{B}_{xHx^{-1}})$ imprimitivity bimodule $\langle \mathscr{L}(\mathscr{B}_{xH}), _{\mathscr{L}(\mathscr{B}_{xHx^{-1}})}[ \ , \ ], [ \ , \ ]_{\mathscr{L}(\mathscr{B}_H)} \rangle$ such that
\begin{equation}\label{wejisdkjsdfiojdsiwerdsxs}
^xS \simeq {\rm{Ind}}(S; \langle \mathscr{L}(\mathscr{B}_{xH}), _{\mathscr{L}(\mathscr{B}_{xHx^{-1}})}[ \ , \ ], [ \ , \ ]_{\mathscr{L}(\mathscr{B}_H)} \rangle).
\end{equation}
We omit the details of the structures of this imprimitivity bimodule because we do not need them in this work.
\end{construction}


\subsubsection{The Results}
\begin{lemma}\label{siokaweiodkzxdaw}
Let $F$ be a $\ast$-ideal of $\mathscr{L}(\mathscr{B})$; and for each $x \in G$ let $J_x$ be the closure of $\{f(x): f \in F\}$. For each $y, z \in G$, we have
\begin{equation*}
B_yJ_z \subset J_{yz}; \ J_zB_y \subset J_{zy}; J_y^{\ast} \subset J_{y^{-1}}.
\end{equation*}
\end{lemma}
\begin{proof}
Let $a \in B_y$ and $b \in \{f(z): f \in F\}$. Let $g \in \mathscr{L}(\mathscr{B})$ such that $g(y)=a$; and we take a $f \in F$ satisfying $f(z)=b$. Let $\{\psi_i\}_{i \in I}$ be an approximation unit of $G$ around $y$; we have
\begin{equation*}
ab={\rm{lim}}_{i \to \infty}\int_G(\psi_i(x)g(x))f(x^{-1}yz))dx=((\psi g)\ast f)(yz) \in J_{yz}.
\end{equation*}
Since $\{f(z): f \in F\}$ is dense in $J_z$, hence we conclude that $B_y J_z \subset J_{yz}$. By the same argument we can prove that $J_z B_y \subset J_{zy}$ and $J_y^{\ast} \subset J_{y^{-1}}$.
\end{proof}

\begin{lemma}\label{eriodkadsdsacx}
We keep the notations of Lemma \ref{siokaweiodkzxdaw}. If $J_e \neq B_e$, then there is a non-degenerate $\ast$-representation $T$ of $\mathscr{B}$ such that $T \neq 0$ and $T(f)=0$ for all $f \in F$.
\end{lemma}
\begin{proof}
By \cite[II.13.18]{MR936628} we can define a topology on the disjoint union of $B_x/J_x$ making $\{B_x / J_x\}_{x \in G}$ a Fell bundle over $G$ such that $x \mapsto f(x)+J_x$ is a continuous cross-section for each $f \in \mathscr{L}(\mathscr{B})$. We denote this Fell bundle by $\mathscr{B}/J$; since $J_e \neq B_e$, hence $\mathscr{B}/J$ is non-zero. By the aid of \cite[II.13.3]{MR936628} it is easy to see that the map $\digamma: \mathscr{B} \to \mathscr{B}/J$ defined by $a \mapsto a+J_{\pi(a)}$ is continuous. Let $S$ be a faithful $\ast$-representation of $\mathscr{B}/J$; then we take $T=S \circ \digamma$. $T$ is a $\ast$-representation of $\mathscr{B}$ satisfying $T(f)=0$ for all $f \in F$; and since $\mathscr{B}/J$ is non-zero, hence $S \neq 0$ and so $T \neq 0$.
\end{proof}

\begin{lemma}
If $\mathscr{B}$ is saturated, then any non-degenerate $\ast$-representation of $\mathscr{D}$ is induced from a non-degenerate $\ast$-representation of $\mathscr{D}_H$.
\end{lemma}
\begin{proof}
Let $\langle T, P \rangle$ be a non-degenerate $\ast$-representation of $\mathscr{D}$. Let $\mathscr{D}'$ be the $G, G/H$ transformation bundle derived from $\mathscr{D}$. Then $\langle \langle T, P \rangle, P \rangle$ is a system of imprimitivity for $\mathscr{D}'$ over $G/H$. Since $\mathscr{B}$ is saturated, hence $\mathscr{D}$ is saturated; by \cite[XI.14.18]{MR936629} $\langle T, P \rangle$ is induced from a $\ast$-representation of $\mathscr{D}_H$. Our proof is complete.
\end{proof}

\begin{lemma}\label{weiodkjscwesasesde}
Let $\langle T,P \rangle$ be a non-degenerate $\ast$-representation of $\mathscr{D}$. If $\langle T, P \rangle$ is induced from a $\ast$-representation of $\mathscr{D}_H$, then it is induced from a $\ast$-representation of $\mathscr{B}_H$ (see Definition \ref{adkjldsalidwxdw}).
\end{lemma}
\begin{proof}
Since $\langle T, P \rangle$ is induced from a $\ast$-representation of $\mathscr{D}_H$, hence there is a projection valued measure $Q$ on $G/H$ such that $\langle \langle T, P \rangle, Q \rangle$ is a system of imprimitivity for $\mathscr{D}$ over $G/H$. Let $\mathscr{D}'$ be the $G, G/H$ transformation bundle derived from $\mathscr{D}$. By \cite[XI.14.17]{MR936629} $\langle \langle T, S \rangle, Q \rangle|F_{\mathscr{D}'}$ (see Construction \ref{eoisdoisdfiocxoierer}) is non-degenerate; hence  the restriction of $\langle \langle T, P \rangle, Q \rangle$ on the linear span of $\{_{\mathscr{L}(\mathscr{D'})}[fr, gs] : f, g \in \mathscr{L}(\mathscr{B}), \ r, s \in \mathscr{C}_0(G/H)\}$ is non-degenerate because this set is dense in $F_{\mathscr{D}'}$ in the inductive limit topology. But we have
\begin{equation}\begin{split}
_{\mathscr{L}(\mathscr{D}')}[fr, gs] (x,yH)&=\int_G fr(yh)(gs)^{\ast}(h^{-1}y^{-1}x)dvh 
\\&=\int_G f(yh)g^{\ast}(h^{-1}y^{-1}x)dvh \ r(yH)s^{\ast}(yH)
\end{split}\end{equation}
$(x,y \in G)$; then by (\ref{sdisdicewewqdsa})
\begin{equation*}
\langle \langle T, S \rangle, Q \rangle(_{\mathscr{L}(\mathscr{D}')}[fr, gs])=Q(rs^{\ast}) \langle T, S \rangle(_{\mathscr{L}(\mathscr{D})}[f,g]),
\end{equation*}
we conclude that $\langle T, S \rangle|F_{\mathscr{D}}$ is non-degenerate. By \cite[XI.14.17]{MR936629}, $\langle T, S \rangle$ is induced from a $\ast$-representation of $\mathscr{B}_H$.
\end{proof}

\begin{construction}
\rm  For each $x \in G$ we let $E_x=\mathscr{C}_0(G/H, D_x)$ (the space of continuous functions from $G/H$ into $D_x$ vanishing at the infinity point). Let $E$ be the disjoint union of $\{E_x\}_{x \in G}$. By \cite[II.13.18]{MR936628} there is a unique topology making $E$ a Banach bundle over $G$ such that for each $\mathfrak{f} \in \mathscr{L}(\mathcal{L})$ and $r \in \mathscr{L}_0(G/H)$ the cross-section $\mathfrak{f}r$ defined by
\begin{equation*}
\mathfrak{f}r(x)(m)=\mathfrak{f}(x)r(m) \ \ \ \ (x \in G; m \in G/H)
\end{equation*}
is continuous. We denote this Banach bundle by $\mathscr{E}$.

For each pair of $\mathfrak{u} \in E_x, \mathfrak{v} \in E_y$, we define $_{\mathscr{D}^{(2)}}[ \ ,\ ]: \mathscr{E} \times \mathscr{E} \to \mathscr{D}^{(2)}$ and $[ \ , \ ]_{\mathscr{D}^{(1)}}: \mathscr{E} \times \mathscr{E} \to \mathscr{D}^{(1)}$ by
\begin{equation*}
_{\mathscr{D}^{(2)}}[\mathfrak{u}, \mathfrak{v}](m)=_{\mathscr{B}^{(2)}}[\mathfrak{u}(m), \mathfrak{v}(yx^{-1}m)]; \ [\mathfrak{u}, \mathfrak{v}]_{\mathscr{D}^{(1)}}(m)=[\mathfrak{u}(xm), \mathfrak{v}(xm)]_{\mathscr{B}^{(1)}}
\end{equation*}
$(m \in G/H)$.

We define right action of $\mathscr{D}^{(2)}$ on $\mathscr{E}$ by
\begin{equation*}
(\psi\mathfrak{u})(m)=\psi(m)\mathfrak{u}(z^{-1}m) \ \ \ \ (m \in G/H)
\end{equation*}
$(\psi \in D_z^{(2)})$ and left action $\mathscr{D}^{(1)}$ on $\mathscr{E}$ by
\begin{equation*}
(\mathfrak{v}\phi)(m)=\mathfrak{v}(m)\phi( y^{-1}m) \ \ \ \ (m \in G/H)
\end{equation*}
$(\phi \in D_z)$. 

It is routing to verify that we have made $\mathscr{E}$ a $\mathscr{D}^{(2)}$-$\mathscr{D}^{(1)}$ equivalence bundle; and if $\mathscr{B}^{(1)}$ and $\mathscr{B}^{(2)}$ are strongly equivalent, then $\mathscr{D}^{(1)}$ and $\mathscr{D}^{(2)}$ are strongly equivalent implemented by $\mathscr{E}$.
\end{construction}

\begin{theorem}\label{weisdklsdfldsfpoepor}
$C^{\ast}(\mathscr{B}_H)$ and $C^{\ast}(\mathscr{D})$ are Morita equivalent implemented by the imprimitivity bimodule $\langle \mathscr{L}(\mathscr{B}), _{\mathscr{L}(\mathscr{D})}[\ , \ ], [ \ , \ ]_{\mathscr{L}(\mathscr{B}_H)}) \rangle$  if and only if $\mathscr{B}$ is $H$-saturated.
\end{theorem}
\begin{proof}
Suppose $\mathscr{B}$ is $H$-saturated. We denote $\mathscr{B}$ by $\mathscr{B}^{(2)}$; and let $\mathscr{B}^{(1)}$ be a saturated Fell bundle over $G$ weakly equivalent to $\mathscr{B}^{(1)}$ implemented by a bundle $\mathcal{L}$.  Let $\mathscr{E}$ be the bundle implementing the weak equivalence of $\mathscr{D}^{(1)}$ and  $\mathscr{D}^{(2)}=\mathscr{D}$. By \cite[XI.14.17]{MR936629} and Theorem \ref{weioweiordeddsaf} each $\ast$-represntation of $\mathscr{B}_H$ is inducible to a non-degenerate $\ast$-representation of $\mathscr{D}$; hence, in order to prove that $\mathscr{B}_H$ and $\mathscr{D}$ are Morita equivalent, by Lemma \ref{esrijodzjodfiojcde} it is sufficient to prove that any non-degenerate $\ast$-representation of $\mathscr{D}$ is induced from a non-degenerate $\ast$-representation of $\mathscr{B}_H$.

Let $R$ be a non-degenerate $\ast$-representation of $\mathscr{D}^{(2)}$. Since $\mathscr{D}^{(2)}$ is $H$-saturated, hence by Lemma \ref{ewioiowekjdsoiwerdcs} $\mathcal{D}$ is $\mathscr{D}^{(2)}$-full for $G/H$. By Theorem \ref{weioweiordeddsaf} $R$ is induced from a non-degenerate $\ast$-representation of $\mathscr{D}^{(2)}_H$; and by Lemma \ref{weiodkjscwesasesde} $R$ is induced from a $\ast$-representation of $\mathscr{B}^{(2)}_H$. The `if' part is proved.

Conversely, let us suppose that $\mathscr{B}$ is not $H$-saturated. Then by Lemma \ref{eriodkadsdsacx} and Construction \ref{eoisdoisdfiocxoierer}, the linear span of $\{_{\mathscr{L}(\mathscr{D})}[f,g]: f, \ g \in \mathscr{L}(\mathscr{B})\}$ is not dense in $C^{\ast}(\mathscr{D})$, so $\mathscr{B}_H$ and $\mathscr{D}$ cannot be Morita equivalent implemented by $\mathscr{L}(\mathscr{B})$.
\end{proof}

\begin{theorem}\label{ewridskjodsklcxwqweqwe}
The following conditions are equivalent:

$(i)$ $\mathscr{B}$ is $H$-saturated;

$(ii)$ $\mathscr{B}_H$ and $\mathscr{D}$ are Morita equivalent implemented by $\mathscr{L}(\mathscr{B})$;

$(iii)$ $\mathscr{B}_H$ and $\mathscr{B}_{xHx^{-1}}$ are Morita equivalent implemented by $\mathscr{L}(\mathscr{B}_{xH})$ for all $x \in G$.
\end{theorem}
\begin{proof}
$(i) \Leftrightarrow (ii)$ is proved in Proposition \ref{weisdklsdfldsfpoepor}. We just need to prove $(i) \Leftrightarrow (iii)$.  

$(i) \Rightarrow (iii)$: By Lemma \ref{esrijodzjodfiojcde}, it is sufficient to prove that if $S$ is a non-degenerate $\ast$-representation of $\mathscr{B}_H$, then ${\rm{Ind}}(S; \mathscr{L}(\mathscr{B}_{xH}))$ exists and is non-zero. But Theorem \ref{erwiosdiojedxsqwde} and (\ref{wejisdkjsdfiojdsiwerdsxs}) ${\rm{Ind}}(S; \mathscr{L}(\mathscr{B}_{xH}))$ exists as a representation of $\mathscr{B}_{xHx^{-1}}$ on the Hilbert space $Y_{xH}$; since $\mathscr{B}$ is $H$-saturated, it is routine to verify that $Y_{xH}$ is a non-zero Hilbert space, hence ${\rm{Ind}}(S; \mathscr{L}(\mathscr{B}_{xH}))$  is nonzero.

$(iii) \Rightarrow (i):$ This may be proved by the same argument of the proof of $(ii) \Rightarrow (i)$ by the aid of Lemma \ref{eriodkadsdsacx}.

\end{proof}

\end{document}